\AtBeginDocument{%
   \def\MR#1{}
}

\documentclass[12pt,reqno]{amsart}
\usepackage{amsmath}
\usepackage{amssymb}
\usepackage[left=1in,top=1.4in,right=1in,bottom=1.2in]{geometry}
\usepackage[usenames]{color}
\usepackage[dvipsnames]{xcolor}
\usepackage{enumerate}
\usepackage[normalem]{ulem}
\usepackage{soul}
\usepackage{graphicx, caption, subcaption}
\usepackage[outdir=./]{epstopdf}
\usepackage[initials]{amsrefs}
\usepackage{hyperref}
\usepackage[utf8]{inputenc}
\usepackage[T1]{fontenc}

\newcommand\Item[1][]{%
  \ifx\relax#1\relax  \item \else \item[#1] \fi
  \abovedisplayskip=0pt\abovedisplayshortskip=0pt~\vspace*{-\baselineskip}}

\theoremstyle{plain}
\newtheorem{theorem}{Theorem}[section]

\newtheorem{claim}[theorem]{Claim}

\newtheorem*{taylor}{Taylor's Theorem}
\newtheorem*{azuma}{Azuma-Hoeffding's Inequality}
\newtheorem*{freedman}{Freedman's Inequality}

\theoremstyle{definition}

\def\a{\alpha}  \def\d{\delta} \def\D{\Delta}
    
\def\G{\Gamma}  \def\k{\kappa} 
     \def\l{\ell}

\def\t{\tau}

\newcommand{\rbrac}[1]{\left(#1\right)}

\newcommand{\abrac}[1]{\left| #1\right|}

\def\E{\mathbb{E}}
\def\G{\mathbb{G}}

\def\Var{\mbox{{\bf Var}}}

\newcommand{\eps}{\varepsilon}

\newcommand{\sqbs}[1]{\left[ #1 \right]}

\newcommand{\Mean}[1]{\E\sqbs{#1}}

\allowdisplaybreaks[1]
\newcommand{\ignore}[1]{}

\newcommand{\beq}[1]{\begin{equation}\label{#1}}
\newcommand{\eeq}{\end{equation}}

\newcommand{\mc}[1]{\mathcal{#1}}
\newcommand{\mct}[1]{\tilde{\mathcal{#1}}}

\title[An introduction to the differential equation method]{A gentle introduction to the differential equation method and dynamic concentration}

\author{Patrick Bennett}
\address{Department of Mathematics, Western Michigan University, Kalamazoo, MI, USA}
\thanks{The first author was supported in part by Simons Foundation Grant \#426894.}
\email{\tt patrick.bennett@wmich.edu}

\author{Andrzej Dudek}
\address{Department of Mathematics, Western Michigan University, Kalamazoo, MI, USA}
\thanks{The second author was supported in part by Simons Foundation Grant \#522400.}
\email{\tt andrzej.dudek@wmich.edu}

\begin{document}

\begin{abstract}
We discuss the differential equation method for establishing dynamic concentration of discrete random processes. We present several relatively simple examples of it and aim to make the method understandable to the unfamiliar reader who has some basic knowledge on probabilistic methods, random graphs and differential equations.
\end{abstract}

\maketitle

\section{Introduction}

The {\em differential equation method} in probabilistic combinatorics is a set of techniques and tools which can be used to analyze random processes that evolve one step at a time. Roughly speaking, the steps should be very small relative to the whole evolving structure, e.g.\ one step could consist of adding one edge to a graph with some large number of vertices. If successful, the method gives us tight bounds on a random variable (or a family of them) that hold at every step of the process, asymptotically almost surely.  (Recall that an event $\mathcal{E}_n$ in a probability space holds \emph{asymptotically almost surely}, abbreviated \emph{a.a.s.}, if $\Pr(\mathcal{E}_n)$ tends to~$1$ as $n$ goes to infinity.) If we manage to prove such bounds on a random variable, we say that we have {\em tracked} that variable. 

The phenomenon that we establish using this method is called {\em dynamic concentration} because at any given step the random variable is concentrated around its expectation, but that expectation is also dynamic in the sense that it changes from one step to the next. We call this dynamic expectation of the random variable the {\em trajectory.} 
The differential equation method is so named for the following reason. Since the steps consist of very small changes relative to the whole structure, even though the steps are discrete, we can essentially treat the process as continuous. The change over one step in the random variable's trajectory can then be approximated by a derivative of a function corresponding to the expected value. As we apply this method several times in this paper, we will see that these one-step changes in trajectories lead us to differential equations, and the fact that our trajectories satisfy them will be crucial to the proof.  We will give some history of the method at the end of this section. 

In this paper we aim to introduce the differential equation method through three examples. The method is not necessarily the best way to analyze the chosen examples, but we would like to illustrate the method in a relatively simple setting. Let us emphasize that all examples are well-known results that can even be considered mathematical folklore.

The first example (see Section \ref{sec:balls}) we will use is the balls and bins problem. This is a classic problem in probability theory that has been extensively studied by several researchers (see, e.g.,~\cite{JK1977}). 
The problem involves $n$ distinguishable bins and $m$ balls. At each step a single ball is placed uniformly at random into one of the bins. We are interested in the number of bins with exactly $k$ balls at the end of the process.  The balls and bins problem is relatively easy because at each step all that happens is some bin with say $j$ balls will now have $j+1$ balls. We will track the number of bins with $0\le k\le\kappa$ balls, for some nonnegative integer~$\kappa$. Our analysis will yield the following:
\begin{theorem}\label{thm:balls}
Fix an integer $\kappa\ge 0$  and a real number $\alpha>0$.  Then, a.a.s.\  in the balls and bins problem with $n$ distinguishable bins and $m = m(n) \le (\frac{1}{2}-\alpha)n\log n $ balls the number of bins with exactly ~$\kappa$ balls is~asymptotically equal to 
\[
\frac{(\frac{m}{n})^\kappa e^{-\frac{m}{n}}}{\kappa!}n.
\]
\end{theorem}
\noindent
 To be clear, when we say that a random variable $X$ is a.a.s.\  asymptotically equal to $x$, we mean that there is some sequence $\eps_n=o(1)$ such that
\[
\lim_{n\rightarrow \infty} \Pr[ (1-\eps_n)x \le X \le (1+\eps_n)x ] =1.
\]

We now describe the second example (see Section \ref{sec:comps}). Let $\G(n,m)$ be the Erd\H{o}s-R\'enyi random graph model that assigns equal probability to all labeled graphs with exactly $m=\lfloor cn\rfloor$ edges, for a positive constant $c$. (For a general introduction to random graphs, see, e.g., \cites{Bollobas2001, JLR2000, FK2016}.)
Our goal is to track the number of components of order~$k$ in $\G(n,m)$ for any $1\le k\le \kappa$, where $\kappa$ is a positive constant. Instead of working with $\G(n,m)$ we will consider an equivalent Erd\H{o}s-R\'enyi process that starts with $n$ vertices and no edges, and at each step adds one new edge chosen uniformly from the set of missing edges. This process is a little bit more complicated than balls and bins, since in one step a component on $j$ vertices can increase its order to $j'$ vertices for any $j'>j$. Roughly speaking, our components of any given order interact with components of all other orders. This is in contrast with balls and bins where the bins with $j$ balls only interact directly with the bins with $j-1$ or $j+1$ balls. However, the analysis of our second example will be only somewhat harder than balls and bins and we will obtain the following: 
\begin{theorem}\label{thm:comps}
Fix an integer $\kappa\ge 1$ and a real number $c>0$. Then,  a.a.s.~the number of components of order~$1\le k\le  \kappa$ in the random graph $\G(n,\lfloor cn\rfloor)$ is ~asymptotically equal to  
\[
\frac{k^{k-2}}{k!} (2c)^{k-1} e^{-2kc}n.
\]
\end{theorem}

We will then do a third example (see Section \ref{sec:match}), which will imply a nontrivial result in extremal graph theory. This will show how certain processes can produce a  structure with interesting extremal properties. Applications of this method have yielded several best known bounds for problems in extremal combinatorics that seem purely deterministic (which is of course one of the most compelling reasons to study probabilistic combinatorics in the first place). Here, we will prove the following:

\begin{theorem}\label{thm:match}
Let $d=d_n$ be a sequence such that $\frac{d}{\log n}\to\infty$. Let $G_n$ be a $d$-regular graph of order $n$. Then the size of a maximum matching in $G_n$ is $\left( \frac12 - o(1) \right) n$.
\end{theorem}
\noindent
In other words, $G_n$ has a matching that is almost perfect. (Here, obviously, we assume that $dn$ is even. So when $d$ is odd we take a limit over even values of $n$.)

The result will follow from our analysis of a random greedy matching process. Specifically, at each step in the process we choose one edge to be in a matching, chosen uniformly at random from all edges that are not incident with any edges already in the matching. The process stops when it is no longer possible to choose any more such edges. 

We now discuss the history of the method. Before the differential equation method came into use by combinatorialists, probabilists knew that certain random processes were concentrated around deterministic functions which could be described by differential equations (see, e.g.\ Kurtz \cite{Kurtz1970}). However, the first result in graph theory whose proof used this phenomenon was by Karp and Sipser \cite{KS1981}, who analyzed a random greedy matching process (more complicated than the one in the above paragraph) when run on a random graph. Their analysis implied that their (fast) algorithm a.a.s.\  outputs a maximum matching in $\G(n, m)$ for certain $m$ in the sparse linear regime. While Karp and Sipser's proof did not look much like the applications of the differential equation method we see today, it did crucially rely on establishing sharp estimates of random variables that change over time, whose trajectory can be given by the solution of a differential equation.  

Ruci\'nski and Wormald \cite{RW1992} were probably the first to prove a result in combinatorics using a method resembling what we present in this paper. They analyzed the \emph{$d$-process}, which is a \emph{constrained random graph process}. It starts with an empty graph and at each step adds one edge chosen uniformly at random from all edges in such a way that the maximum degree of the induced graph is always bounded above by~$d$. In particular, they obtained bounds on the probability of unlikely events using McDiarmid's inequality, a concentration inequality that resembles the ones we will use (Azuma-Hoeffding's inequality and Freedman's inequality). 

For a broad survey of the relatively early applications of the method, see Wormald~\cite{W1999}. There you will also find Wormald's ``black box'' differential equation method theorem, which asserts that if a family of random variables satisfies a list of conditions then they are all dynamically concentrated around their trajectories. (Recently Warnke \cite{Warnke2020} gave a stronger version of Wormald's  theorem.) The method owes much of its early development, particularly in the 1990s, to Wormald. His black box theorem does not apply to every interesting process, but it does apply to many (it can apply to our first two examples in this paper, and see \cite{DM2010, W1999, BBal2016} for several more examples). Wormald's survey also contains several examples where his black box theorem does not apply but he is nevertheless able to analyze them. For more of Wormald's contributions see also~\cite{W1995, W2003}.

It is also worth mentioning that the differential equation method can be also viewed as a more precise version of the R\"odl nibble~\cite{Rodl1985}. Whereas a typical application of the differential equation method considers a process that progresses one small random step at a time, the nibble method considers a {\em semi-random} process, which means a process that progresses by iteratively choosing many possible random steps to take and then taking some subset of those steps. This approach was first successfully used by R\"odl~\cite{Rodl1985}, who proved the Erd\H{o}s and Hanani conjecture. Another spectacular application of this technique was given by Kim~\cite{Kim1995}, who found the right order of magnitude of the Ramsey number $R(3,t)$. For a more recent application of the nibble technique see Guo and Warnke's paper on packing $R(3, t)$ graphs~\cite{GW2020}.

Since the late 2000s, Bohman and others have made several important contributions to the method, including the application of concentration inequalities that had previously been unused in the differential equation method. Bohman \cite{B2009} and subsequently Bohman and Keevash \cite{BK2010, BK2013, BK2020} improved several best-known bounds in extremal graph theory, including several Ramsey and Tur\'an bounds. Perhaps the most well known of these results is that the Ramsey number $R(3, t) \ge (1/4 + o(1)) t^{3/2} / \log t$, which was proved independently by Fiz Pontiveros, Griffiths and Morris \cite{FGM2020} and Bohman and Keevash \cite{BK2013, BK2020}. Both proofs are analyses of the triangle-free process (see \cite{B2009} for a more gentle treatment of this process). This is also an example of a constrained random graph process, more specifically it adds one edge at a time chosen uniformly at random from all edges whose addition to the current graph would maintain that it is triangle-free. The process halts with a maximal triangle-free graph, and the very detailed analyses in \cite{BK2013, BK2020, FGM2020} yield a bound on the independence number on the final graph, implying the $R(3, t)$ bound. Roughly speaking, they are able to apply the differential equation method to track the upper bound on the independence number of the evolving graph.
This result is a major breakthrough that required several new innovations, possibly the most important of which is that they tracked a much larger family of random variables than what Bohman originally considered in~\cite{B2009}. This larger family of variables was carefully chosen to give the information needed to get such a good bound on the independence number of the graph. 

Finally, let us emphasize that there are other introductory papers describing the differential equation method. Without any doubt the most fundamental is the aforementioned survey by Wormald~\cite{W1999} (and also \cite{W1995, W2003}). %Another excellent paper that can served as an (advanced) introduction to the differential equation method is due to Warnke~\cite{Warnke2014b}.
Readers who are more interested in the computer science aspects of the differential equation method can find a useful survey of D\'{i}az and Mitsche~\cite{DM2010}.

\bigskip
\textbf{Acknowledgement.} We are very grateful to all our colleagues for helpful discussions and useful comments. These include Bethany Austhof, Nick Christo, Sean English, Alan Frieze, Andrzej Ruci\'nski, Lutz Warnke and Nick Wormald. We would also like to thank the anonymous referees for many helpful comments and suggestions. 

%This command prints all references
%\nocite{*} 

\section{Preliminaries}
In this section we present a few tools that we are going to use in the rest of the paper. We start with the well known Taylor's theorem from Calculus.
As it will be explained in the next sections, tracking random variables will require to approximate them by continuous functions. Therefore, Taylor's theorem will be very handy.
\begin{taylor}
Let $f:\mathbb{R}\to\mathbb{R}$ be a function twice differentiable on the closed interval $[a,b]$. Then, there exists a number $\tau$ between $a$ and $b$ such that 
\begin{equation}\label{eq:taylor}
f(b) = f(a) + f'(a)(b-a) + \frac{f''(\tau)}{2} (b-a)^2.
\end{equation}
\end{taylor}

In order to estimate a failure probability we will make a use of Azuma-Hoeffding's inequality. For that we need to define supermartingales. Consider a random process (a sequence of random structures) such as a sequence of graphs or configurations of balls in bins. Let $Y_0,Y_1,Y_2,\dots,Y_i,\dots$ be random variables counting something in the $i^{th}$ structure. Let $\mc{F}_i$ be the history of the process up to step~$i$. More formally $\mc{F}_i$ is a partition of the underlying probability space (the set of all possible sequences of structures) where two sequences are in the same part if they agree on the first~$i$ structures.  In particular, conditioning on $\mc{F}_i$ tells us the current structure and the value of~$Y_i$. For convenience, we also define a \emph{difference operator} $\D$ as \[
\D Y_i = Y_{i+1}-Y_i.
\]
We say that the sequence $Y_0,Y_1\ldots$ is a \emph{supermartingale} if for each $i\ge 0$, $\Mean{\D Y_{i} | \mc{F}_{i}}\le 0$.
Similarly, we define a \emph{submartingale} if the sequence satisfies, $\Mean{\D Y_{i} | \mc{F}_{i}}\ge 0$.

%More generally and formally, suppose $\mc{F}_0 \subseteq \mc{F}_1\subseteq\ldots$ is a sequence of partitions of a probability space such that $Y_i$ is {\em $\mc{F}_i$-measurable} (in other words, $Y_i$ is a constant on each part of the partition $\mc{F}_i$). Then we say $Y_0,Y_1\ldots$ is a supermartingale if for each $i\ge 0$, $\Mean{\D Y_{i} | \mc{F}_{i}}\le 0$. In applications of the differential equations method, the probability space is the set of all possible sequences of steps the process could take, and the history $\mc{F}_i$ is just the partition where two such sequences of steps are in the same part if they agree on steps $1, \ldots i$. 

\begin{azuma}[\cites{A1967, H1963}]
If $Y_0,Y_1,\dots$ is a supermartingale and $|Y_j - Y_{j-1}|\le C$ almost surely for all $j\ge 1$, then for all positive integers $m$ and all positive reals~$\lambda$,
\begin{equation}\label{eq:azuma}
\Pr(Y_m - Y_0 \ge \lambda) \le \exp\left( -\frac{\lambda^2}{2C^2 m} \right).
\end{equation}
\end{azuma}
When we will apply Azuma-Hoeffding's inequality we  first define two sequences $Y_0^\pm,Y_1^\pm,\dots$, where $Y_i^+$ is a supermartingale and $Y_i^-$  is a submartingale. Next apply~\eqref{eq:azuma} to $Y_i^+$ with $\lambda = -Y_0^+$, where $Y_0^+<0$. This will imply that 
\[
\Pr(Y_m^+\ge 0) = \Pr(Y_m^+ - Y_0^+ \ge \lambda) \le \exp\left( -\frac{\lambda^2}{2C^2 m} \right).
\]
Consequently, for $\lambda^2 \gg C^2 m$, we get that a.a.s.\  $Y_m^+ < 0$. Similarly, since $-Y_i^-$ is also a supermartingale, \eqref{eq:azuma} will yield that a.a.s.\  $Y_m^- >0$.

As we will see in Section~\ref{sec:match} sometimes~\eqref{eq:azuma} is not strong enough. For variables, which experience relatively large but rare one-step changes, the following deviation inequality is often helpful.
\begin{freedman}[\cite{F1975}]
Suppose $Y_0,Y_1,\dots$ is a supermartingale such that $Y_j - Y_{j-1} \le C$ for all $j$, and let $V_m =\displaystyle \sum_{k \le m} \Var[ \Delta Y_k| \mathcal{F}_{k}]$.  Then, for all positive reals~$\lambda$,
\begin{equation}\label{eq:freedman}
\Pr(\exists m: V_m \le b \text{ and } Y_m - Y_0 \geq \lambda) \leq \displaystyle \exp\left(-\frac{\lambda^2}{2(b+C\lambda) }\right).
\end{equation}
\end{freedman}
Freedman's inequality still has the parameter $C$, which is now only an upper bound on the one-step change (and not its absolute value). 
So what do we gain? The benefit is that Freedman's inequality has a dependence on $C$ that is not as bad as in Azuma-Hoeffding's inequality. What really matters for Freedman's inequality is the one-step variance. Roughly speaking, the one-step variance in Freedman plays the same role as $C^2$ in Azuma-Hoeffding's inequality. For the type of variable we would like to track, the one-step variance is much smaller than $C^2$ because a change as large as $C$ is a rare event. 

Throughout the paper, we will use standard big-O and little-o notation, and all asymptotics are as $n \rightarrow \infty.$ We write $a(n) \sim b(n)$ if $a(n)=(1+o(1))b(n)$ and $a(n) \gg b(n)$ if $b(n) = o(a(n))$.
We also write $a(n) = \Omega(b(n))$ if there exists a (possibly small) positive constant $\gamma$ such that $a(n) \ge \gamma b(n)$.

We abuse notation by ignoring issues of integrality (say we are omitting integer floors). All logarithms are natural, that is to base $e$.

\section{Balls and bins}\label{sec:balls}

%\subsection{Problem}\label{sec:bins}
%
%The balls into bins problem is a classic problem in probability theory and was intensively studied in the literature (see, e.g.,~\cite{JK1977}). 
%The problem involves $n$ bins and $m$ balls. Each time, a single ball is placed uniformly at random into one of the bins. After all balls are in the bins, we look at the number of bins having exactly $\kappa$ balls, for some nonnegative integer~$\kappa$.

In this section we prove Theorem~\ref{thm:balls}. Recall that we have $n$ distinguishable bins all starting out empty, and at each step $i$ we place one ball into a uniformly random bin (independently from other balls). Let $X_k(i)$ be the number of bins with exactly $k$ balls at step~$i$, where $0\le k\le \kappa$ and $0\le i\le m$. In other words,  we count the number of bins with $k$ balls after $i$ balls were placed. We will show that this random variable $X_k(i)$ is highly concentrated, and hence $X_k(i)$ follows some trajectory. 

In Section~\ref{sec:bins:heuristic} we show how to guess the anticipated trajectory of $X_k(i)$. In Section~\ref{sec:bins:rigorous} we provide a rigorous argument working under a slightly weaker assumption that the total number of balls $m=m(n)$ is at most $\left(\frac{1}{12}-\alpha\right)n\log n$, where $\alpha>0$. Finally, in Section~\ref{sec:bins:self} we improve the upper bound on $m$ to $m\le \left(\frac{1}{2}-\alpha\right)n\log n$.

\subsection{Determining trajectories heuristically}\label{sec:bins:heuristic}
For this simple example, there are several ways to heuristically derive the trajectories. The easiest way is to directly estimate the expected value of $X_k(i)$ after $i$ balls are placed. Observe that the number of balls in a fixed bin is binomially distributed and hence the probability that a fixed bin has exactly $k$ balls is
\[
\binom{i}{k} \rbrac{ \frac 1n}^k \rbrac{1-\frac 1n}^{i-k}.
\]
Thus, the expected number of bins with $k$ balls is
\[
n\binom{i}{k} \rbrac{ \frac 1n}^k \rbrac{1-\frac 1n}^{i-k}.
\]
Since $1\ll i$ and $k$ is a fixed number, the last factor is asymptotically equal to $\rbrac{1-\frac 1n}^{i}$. Furthermore, since $i = o(n^2)$, the expansion of the $\log$ function implies that $\rbrac{1-\frac 1n}^{i} \sim e^{-\frac{i}{n}}$. Finally, we obtain that the expected number of bins with $k$ balls is asymptotically equal to
\[
\frac{\left(\frac{i}{n}\right)^k e^{-\frac{i}{n}}}{k!}n.
\]

However, often in applications of the differential equation method it is not so easy to derive the trajectories. So we will show another way that is more representative of harder examples. We will calculate the one-step changes in our variables, interpret those as a system of differential equations, and then solve that system to derive the trajectories. 

Let $\mc{F}_i$ be the history of the process up to step $i$. 
We calculate the expected one-step change of $X_k(i)$, conditional  on 
 $\mc{F}_{i}$, namely, 
\[
\Mean{\D X_k(i) | \mc{F}_{i}} = \Mean{X_k(i+1)-X_k(i) | \mc{F}_{i}}.
\]
Assume that we have already placed $i$ balls. From now on we will often suppress the $i$ in $X_k(i)$, unless we want to emphasize it or the argument is something other than $i$. Now we place a new ball uniformly at random to one of the $n$ bins. If it is placed into one of the bins with $k$ balls (which happens with probability $X_k/n$), then the change in $X_k$ will be $-1$.  However, provided that $k\ge 1$, if it goes into one of the bins with $k-1$ balls  (which happens with probability $X_{k-1}/n$), then the change in $X_k$ will be $1$. Thus,
\begin{equation}\label{eq:X_k}
\Mean{\D X_k(i) | \mc{F}_{i}} =  -\frac{X_k}{n} + \frac{X_{k-1}}{n},
\end{equation}
where, for convenience, we define $X_{-1}= 0$.

Assume heuristically that
\[
X_k(i) \sim nx_k(i/n)
\] 
for some deterministic function $x_k$. Observe that we must have $x_0(0)=1$ and $x_k(0)=0$ for $k\ge 1$, since $X_0(0)=n$ and $X_k(0)=0$ for $k\ge 1$.

Now we introduce a time parameter
\[
t = t_i: = \frac{i}{n}.
\] 
We will often suppress the $i$ in $t_i$ unless we want to emphasize it or the subscript is something other than $i$. We think of $t$ as a \emph{continuous} time parameter (e.g.\ we take derivatives with respect to $t$).

We estimate $\Mean{\D X_k(i) | \mc{F}_{i}}$ by using Taylor's theorem (c.f.~\eqref{eq:taylor}). We heuristically assume  that $x_k(t)$ is twice differentiable. Thus,
\[
X_k(i+1) - X_k(i) \sim n\big(x_k(t_{i+1}) - x_k(t_i)\big) = n\left(x_k'(t_{i})\cdot \D t + \frac{x_k''(\t)}{2} (\D t)^2\right), 
\]
where $\D t = t_{i+1}-t_i = 1/n$ and $\t$ is some number in $[t_i, t_{i+1}]$. If we assume that $x_k''(\t) = o\big{(}n x_k'(t_i)\big{)}$, then 
\[
\Mean{\D X_k(i) | \mc{F}_{i}} \sim n\rbrac{\frac{x_k'(t)}{n} + \frac{x_k''(\t)}{n^2}} \sim x_k'(t).
\]
This happens all the time in the differential equation method: we approximate the expected one-step change in a random variable by the derivative of its trajectory (sometimes multiplied by an appropriate scaling factor).  
Now looking back to \eqref{eq:X_k}, we see that the left-hand side is approximately $x_k'$ and the right-hand side is 
\[
-\frac{X_k}{n} + \frac{X_{k-1}}{n} \sim -x_k + x_{k-1},
\]
and thus we have derived the following system of differential equations:
\[
x_k' = -x_k + x_{k-1}
\]
with $0\le k\le \kappa$, $x_{-1}:=0$ and initial conditions $x_0(0)=1$ and $x_k(0)=0$ for $k\ge 1$. %Note that if one knows $x_{k-1}$ one can then solve for $x_k$ since the above is just a first order linear differential equation in $x_k$.
% \[
% (x_k\circ t)(i+1) - (x_k\circ t)(i) = (x_k\circ t)'(i)+\frac{(x_k\circ t)''(\eta)}{2}
% \]
% for some $\eta\in[i-1,i]$
% and the chain rule, we get
% \begin{equation}\label{eq:taylor}
% x_k(t_{i+1}) - x_k(t_i) = \frac{x_k'(t)}{n} + \frac{x_k''(\t)}{2n^2} \sim \frac{x_k'(t)}{n},
% \end{equation}
% if we assume that $\frac{x_k''(\t)}{2n^2} = o\big{(}\frac{x_k'(t)}{n}\big{)}$. Consequently, 
% \[
% x_k'(t) \sim -x_k(t) + x_{k-1}(t),
% \]
% where $x_{-1}(t)=0$.
% So this can be viewed as a system of $\kappa+1$ first-order linear differential equations
% \[
% x_k' = -x_k + x_{k-1}
% \]
% with $0\le k\le \kappa$ and initial conditions $x_0(0)=1$ and $x_k(0)=0$ for $k\ge 1$.
It is easy to see that the unique solution to this system is given by 
\[
x_k(t) = \frac{t^k e^{-t}}{k!}.
\]
Indeed, for $k=0$ we get $x_0' = -x_0$ with $x_0(0)=1$ implying $x_0 = e^{-t}$. For $k\ge 1$ we solve the system iteratively. It suffices to observe that knowing $x_{k-1} = \frac{t^{k-1} e^{-t}}{(k-1)!}$ and $x_k(0)=0$ yield $x_k = \frac{t^k e^{-t}}{k!}$ as the unique solution of the first-order linear differential equation \linebreak $x_k' = -x_k + \frac{t^{k-1} e^{-t}}{(k-1)!}$.

%\begin{rem}
%One can also guess $x_k(t)$ as follows. The probability that a given bin has exactly $k$ balls after $tn$ balls were placed is 
%\[
%\binom{tn}{k} \left(\frac{1}{n}\right)^k \left(1-\frac{1}{n} \right)^{tn-k}
%\sim \frac{(tn)^k}{k!} \left(\frac{1}{n}\right)^k e^{-\frac{tn-k}{n}}
%\sim \frac{t^k e^{-t}}{k!},
%\]
%since $k$ is fixed. So the expected number of bins with exactly $k$ balls is about $n \frac{t^k e^{-t}}{k!}$. Our goal is to show that $X_k$ is highly concentrated around its mean. Therefore, we are guessing that $X_k(i) \sim n \frac{t^k e^{-t}}{k!}$, where $t=\frac{i}{n}$.
%\end{rem}

\subsection{Rigorous argument}\label{sec:bins:rigorous}
Recall that our goal is to show that a.a.s.\  $X_k(i) \sim nx_k(t)$. But in light of equation \eqref{eq:X_k}, we cannot track $X_{k}$ unless we also track $X_{k-1}$, which forces the tracking of $X_{k-2}(i)$, etc. This illustrates a common issue in the method: our family of tracked variables must form a {\em closed system} in the sense that we must be able to write (at least approximately) the expected one-step change of any of our variables in terms of other variables in our tracked family.  Thus, we need to show that a.a.s.\  for each $0\le k\le \kappa$,
\begin{equation}\label{eq:bb_goal}
n\big(x_k(t) - \eps(t)\big) \le X_k(i)\le n\big(x_k(t) + \eps(t)\big)
\end{equation}
for some error function $\eps(t)$ satisfying $x_k(t) \gg \eps(t)$. (For simplicity we will choose the same error function for each $k$.)
In Figure \ref{fig:opt}(\textsc{\subref{subfig:1}}) we illustrate how $X_1(i)$ stays within a narrow envelope around its trajectory.

\begin{figure}
\centering
\begin{subfigure}[b]{0.4\textwidth}
                \centering
                \includegraphics[width=\textwidth]{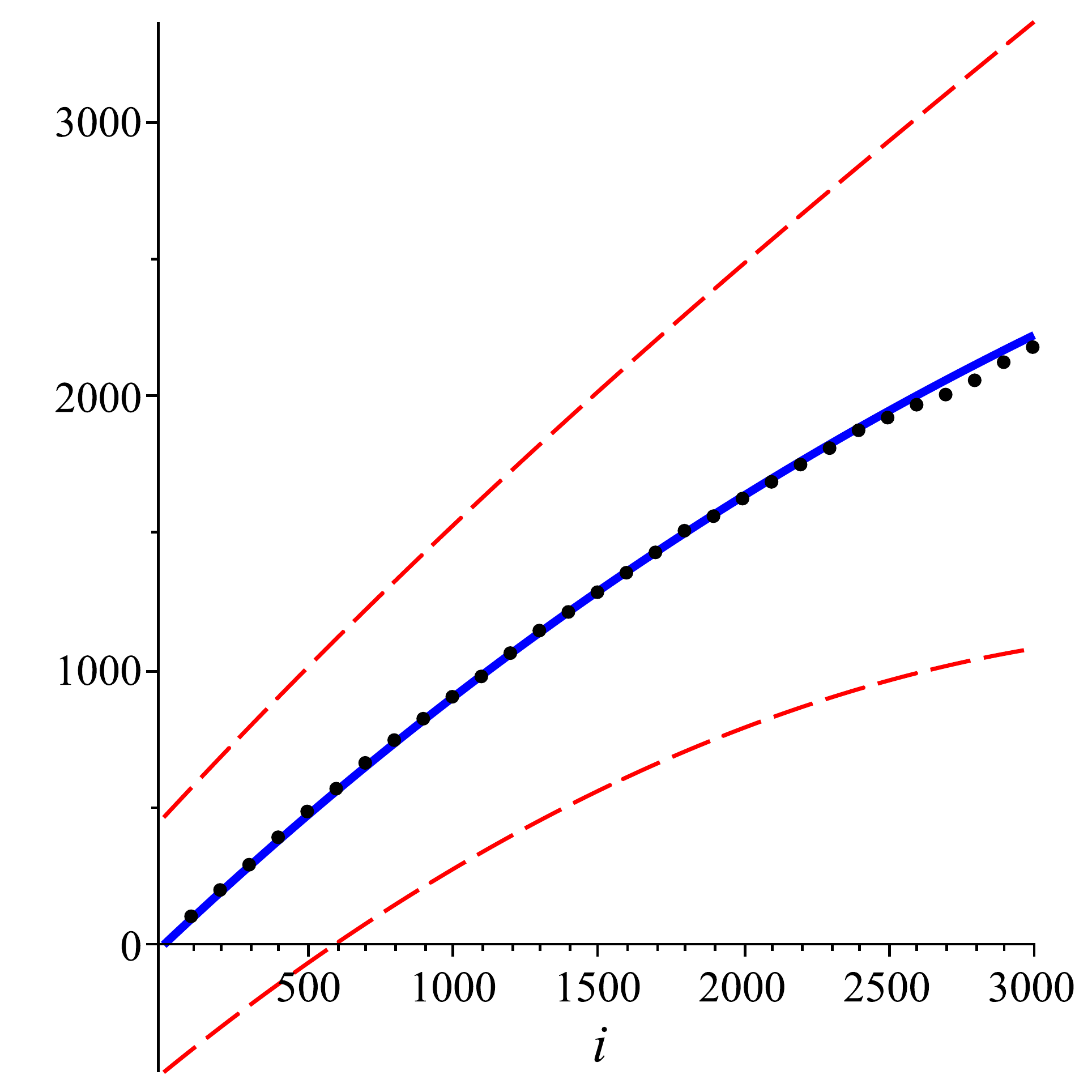}
                \caption{}
                \label{subfig:1}
\end{subfigure}
\quad
\begin{subfigure}[b]{0.4\textwidth}
                \centering
                \includegraphics[width=\textwidth]{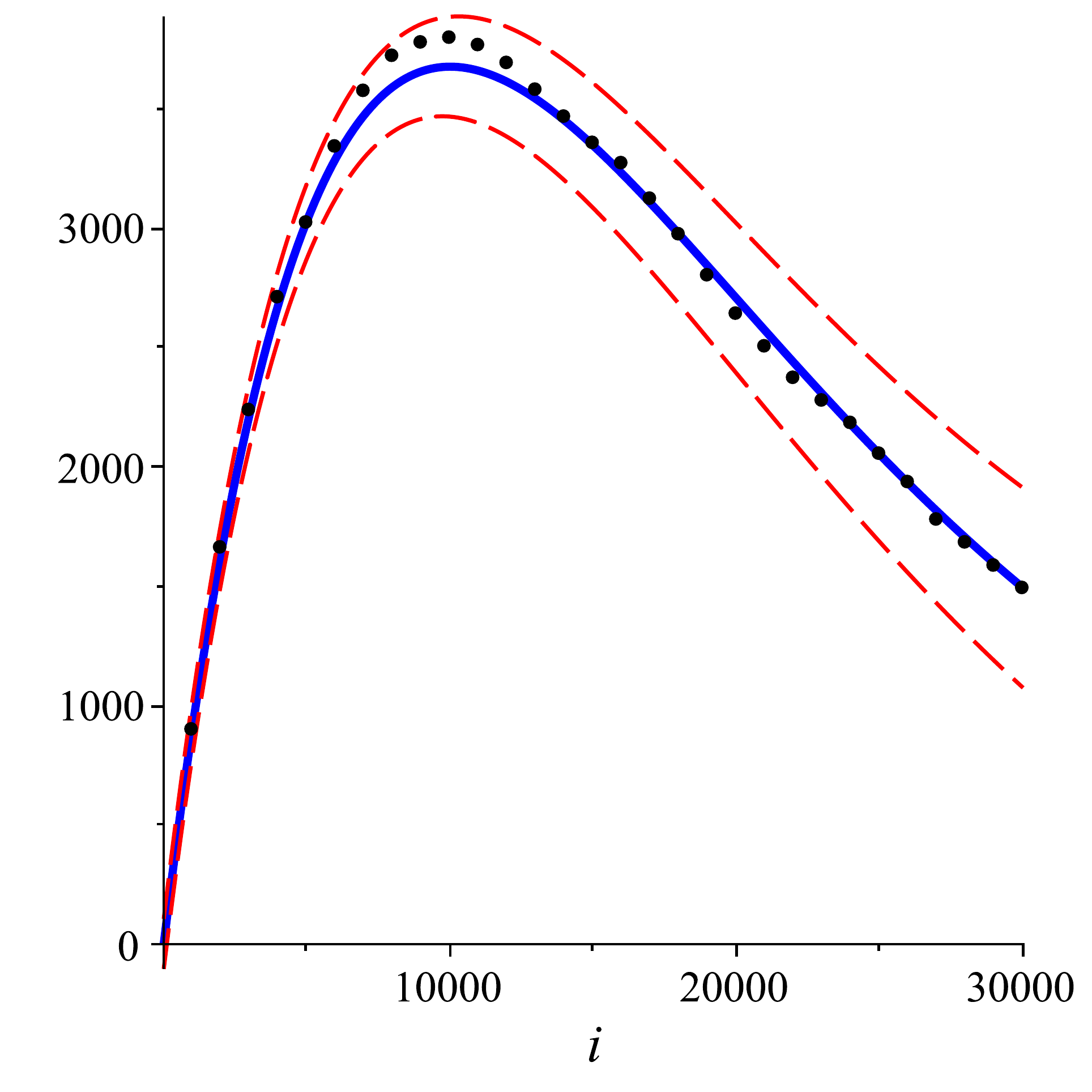}
                \caption{}
                \label{subfig:2}
\end{subfigure}%
\caption{The anticipated trajectory of the random variable $X_1(i)$ from Section~\ref{sec:bins:heuristic} (the bold blue curve), which counts the number of bins with exactly one ball in the balls and bins problem with $10,\!000$ bins and $100,\!000$ balls. 
The trajectory and the simulated values of $X_1(i)$ (black dots) are trapped between error functions (the red dashed curves)
without (\textsc{\subref{subfig:1}}) and with self-correction~(\textsc{\subref{subfig:2}}). On (\textsc{\subref{subfig:1}}) and (\textsc{\subref{subfig:2}}) there are parts of the process for $1\le i\le 3,\!000$ and $1\le i\le 30,\!000$, respectively.}
\label{fig:opt}
\end{figure}

For each step $i'$ let $\mc{E}_{i'}$ be the event that \eqref{eq:bb_goal} holds for all $i \le i'$. Essentially, $\mc{E}_{i'}$ is a ``good event'' stipulating that all our random variables are approximately what we expect them to be. We show that it is very unlikely to stray outside the good event. At this point it may be helpful to think of the method as a probabilistic version of an induction proof. In this analogy, the good event is the induction hypothesis. In a way, we are proving that when this ``induction hypothesis'' holds at step $i$, it is very likely to keep holding at step $i+1$. More precisely, if we were to show that
\begin{equation}\label{eqn:lowbound}
\Pr(\text{\eqref{eq:bb_goal} fails at step $i$}\ |\ \mc{E}_{i-1}) = o(1/m(n)),
\end{equation}
we could conclude that
\begin{equation}\label{eqn:unionbd}
   \Pr(\text{\eqref{eq:bb_goal} does not hold for some $0\le i\le m(n)$})
\le \sum_{i=0}^{m(n)} \Pr(\text{\eqref{eq:bb_goal} fails at step $i$}\ |\ \mc{E}_{i-1}) = o(1). 
\end{equation}
\noindent Technically, though, the differential equation method typically does not use the union bound in \eqref{eqn:unionbd}, and so it is not necessary to establish \eqref{eqn:lowbound}. Indeed, the method we will use give an upper bound on the probability that~\eqref{eq:bb_goal} does not hold for some $0\le i\le m(n)$ directly without any union bound over $i$. However the proof still feels much like a ``probabilistic induction'' in the sense that part of our argument will involve calculations under the assumption that the good event holds at the previous step.

Define two types of random variables $X_k^+$ and $X_k^-$,
\[
X_k^\pm=X_k^\pm(i)=
\begin{cases} 
X_k(i) - \big(x_k(t)  \pm \eps(t)\big)n & \mbox{if $\mc{E}_{i-1}$ holds,}\\
X_k^\pm (i-1) & \mbox{otherwise.}
\end{cases}
\]
We assume here that $\mc{E}_{-1}$ is the trivial event that holds with probability~1. Observe that, in particular, $X_k^\pm(0)=\mp\eps(0)n$.

We ``freeze'' the $X_k^\pm(i)$ variables outside $\mc{E}_{i-1}$ (i.e.\ if the good event ever fails these variables will remain constant from that step onward).
%We want to choose the error function in such a way that the variables $X_k^+(i)$ are supermartingales (i.e.\ $\Mean{\D X_k^+ (i)| \mc{F}_{i}}\le 0$) and $X_k^-(i)$ are submartingales (i.e.\ $\Mean{\D X_k^- (i)| \mc{F}_{i}}\ge 0$). This together with Azuma-Hoeffding's inequality will show that the bounds in~\eqref{eq:bb_goal} hold. 
So if $\mc{E}_{i}$ fails, the one-step change in $X_k^\pm(i)$ is always precisely zero there and so they are super/submartingales, as required. Thus it remains to show that they are also super/submartingales when $\mc{E}_{i}$ holds. 

To establish that $X_k^+ (i)$ is a supermartingale, let us check that $\Mean{\D X_k^+ (i)| \mc{F}_{i}, \mc{E}_i}\le 0$ (here we are conditioning on $\mc{F}_i$ and $\mc{E}_i$, so we assume we are in some part of the partition $\mc{F}_i$ such that $\mc{E}_i$ holds). First observe that by~\eqref{eq:X_k} and~\eqref{eq:bb_goal} we get 
\begin{align}
\Mean{\D X_k(i) | \mc{F}_{i},\mc{E}_i} &=  -\frac{X_k}{n} + \frac{X_{k-1}}{n}\nonumber\\
&\le -\big(x_k(t) - \eps(t)\big) + \big(x_{k-1}(t) + \eps(t)\big)=-x_k(t) + x_{k-1}(t) + 2\eps(t). \label{eqn:noself}
\end{align}
Furthermore, since it is easy to check that $|x_k''(u)| = O(1)$ uniformly for all $u$, we get by Taylor's theorem that
\[
x_k(t_{i+1}) - x_k(t_i) = \frac{x_k'(t_i)}{n} + O\left(\frac{1}{n^2}\right),
\]
and similarly
\[
\eps(t_{i+1}) - \eps(t_i) = \frac{\eps'(t_i)}{n} + O\left( \frac{|\eps''(\t)|}{n^2} \right)
\]
for some $\t\in[t_i,t_{i+1}]$.
At this moment we do not know what $\eps(t)$ is, therefore, we are unable to bound its second derivative.

Now  the above calculations imply that
\begin{align*}
\Mean{\D X_k^+ (i)| \mc{F}_{i}, \mc{E}_i} &\le -x_k(t) + x_{k-1}(t) + 2\eps(t)\\
&\qquad  -x_k'(t) + O\left(\frac{1}{n}\right) - \eps'(t) + O\left( \frac{|\eps''(\t)|}{n} \right)\\
&=2\eps(t) - \eps'(t) +  O\left( \frac{|\eps''(\t)|+1}{n} \right),
\end{align*}
since $x_k'(t) = -x_k(t) + x_{k-1}(t)$. So the $X_k^+(i)$ variables are supermartingales if
\begin{equation}\label{eq:super1}
2\eps(t) - \eps'(t) +  O\left( \frac{|\eps''(\t)|+1}{n} \right) \le 0.
\end{equation}

Next we will use the Azuma-Hoeffding inequality. Clearly, $|\D X_k(i)| \le 1$. It is not difficult to show that $|n\D x_k(t_i)| \le 1$. First, one can verify that $x_k(u)$ has a unique maximum on $u\ge 0$ at $u=k$, which satisfies $x_k(k) < 1/2$. Next,  Taylor's theorem implies that
\[
|n\D x_k(t_i)| = n\left|\frac{x_k'(t_i)}{n} + O\left(\frac{1}{n^2}\right)\right|  = x_k'(k) + O\left(\frac{1}{n}\right) \le 1.
\]
Consequently, 
\[
|\D X_k^+ (i)| \le |\D X_k(i)| + |n\D x_k(t_i)| + |n\D \eps(t_i)| \le 1+1+|n\D \eps(t_i)|.
\]
We will choose $\eps(t)$ in such a way that 
\begin{equation}\label{eq:super2}
|n\D \eps(t_i)| = O(1).
\end{equation}
%and thus $C = O(1)$.
Note that if $\mc{E}_{i'}$ fails at any step $i' \le m$ then we either have $X_k^+(i')>0$ or $X_k^-(i') <0$ for some $k$, and since the variables $X_k^\pm$ are frozen outside the good event we would then have $X_k^+(m)>0$ or $X_k^-(m) <0$. Thus
\begin{align*}
\Pr\big(\exists\, i' \le m: \mc{E}_{i'} \mbox{ fails}\big) &\le \sum_{0 \le k \le \k} \Pr(X_k^+(m) >0) + \Pr(X_k^-(m) <0),
\end{align*}
which explains how we will avoid using the union bound from line \eqref{eqn:unionbd}. We will show that the last expression is $o(1)$. Inequality~\eqref{eq:azuma} applied with $C = O(1)$ will yield that 
\begin{align*}
    \Pr(X_k^+(m) >0) = \Pr(X_k^+(m) - X_k^+(0) > n\eps(0)) \le \exp\left( -\frac{(n\eps(0))^2}{2C^2 m} \right).
\end{align*}
In order to conclude that the latter is $o(1)$ we need
\begin{equation}\label{eq:super3}
n^2 \eps^2(0) \gg m.
\end{equation}
This will imply that 
\[
\sum_{0 \le k \le \k} \Pr\big( X_k^+(m) >0\big)
\le (\k+1)\Pr\big(X_k^+(m) >0)\big) = o(1).
\]
We will similarly see that $\sum_{0 \le k \le \k} \Pr\big( X_k^-(m) <0\big) = o(1)$, and we will justify that after a little discussion of the error function $\eps(t)$.

Finding the right error functions is in general non-trivial.
Observe that by taking for example
\begin{equation}\label{eq:bb_eps}
\eps(t) = n^{-1/3} e^{3t} \quad \text{and}\quad m \le \frac{1}{9}n\log n
\end{equation}
all conditions \eqref{eq:super1}, \eqref{eq:super2} and \eqref{eq:super3} are satisfied. Indeed, conditions~\eqref{eq:super1} and~\eqref{eq:super3} are easy to verify. For~\eqref{eq:super2} observe that by Taylor's theorem for some $\t\in[t_i,t_{i+1}]$
\[
|n\D \eps(t_i)| = |\eps'(t_i)| + O\left( \frac{|\eps''(\t)|}{n} \right)
= 3n^{-1/3}e^{3t_i} + O\left( \frac{e^{3\t}}{n^{2/3}} \right).
\]
By assumption $i\le \frac{1}{9}n\log n$, which yields $t_i\le \frac{1}{9}\log n$ and so
\[
|n\D \eps(t_i)| \le 3n^{-1/3}e^{3 \frac{1}{9}\log n} + O\left( \frac{e^{3 \frac{1}{9}\log n}}{n^{2/3}} \right)
= 3 + O(n^{-1/3}) = O(1).
\]

One can also derive symmetric calculations for $X_k^-(i)$ as follows. Since
\begin{align*}
\Mean{\D X_k(i) | \mc{F}_{i}, \mc{E}_i} &=  -\frac{X_k}{n} + \frac{X_{k-1}}{n}\\
&\ge -\big(x_k(t) + \eps(t)\big) + \big(x_{k-1}(t) - \eps(t)\big)\\
&=-x_k(t) + x_{k-1}(t) - 2\eps(t), 
\end{align*}
we get
\[
\Mean{\D X_k^-(i)| \mc{F}_{i}, \mc{E}_i} \ge-2\eps(t) + \eps'(t) +  O\left( \frac{|\eps''(\t)|+1}{n} \right) \ge 0
\]
due to the choice of~\eqref{eq:bb_eps} implying that the variables $X_k^-(i)$ are submartingales and also
\[
|\D X_k^- (i)| \le |\D X_k(i)| + |n\D x_k(t_i)| + |n\D \eps(t_i)| =O(1).
\]
Now the Azuma-Hoeffding inequality applied to the supermartingale $-X_k^-(i)$ yields
\begin{align*}
\Pr(-X_k^-(m) >0) = \Pr\big((-X_k^-(m)) - (-X_k^-(0)) > n\eps(0)\big) \le \exp\left( -\frac{(n\eps(0))^2}{2C^2 m} \right) = o(1).
\end{align*}

Finally, notice that in~\eqref{eq:bb_goal} we want to maximize $t$ having $x_k(t) \gg \eps(t)$, i.e., 
\[
\frac{t^k e^{-t}}{k!} \gg n^{-1/3} e^{3t},
\]
which is true as long as, for example, $1\ll t\le \rbrac{\frac{1}{12} - \a} \log n$ for some $\a>0$. This implies that we can take any $m\le \left(\frac{1}{12}-\alpha\right)n\log n$ as the total number of balls, in which case the number of bins with exactly $\kappa$ balls is $X_\kappa(m) \sim n\frac{t^k e^{-t}}{k!}$ a.a.s., where $t = \frac{m}{n}$.

Of course the constant ``$1/12$'' can be easily improved by optimizing our argument but it is not difficult to see that in any case, without some new idea, $t$ must always be smaller than $\frac{1}{6}\log n$ implying that $m$ must be smaller than $\frac{1}{6} n\log n$. In the next section we describe how to get better error bounds, resulting in a much better constant.

%Of course one can do it more carefully and choose for $k\ge 1$,
%\[
%\eps(t) = \omega \left(\frac{\log n}{n}\right)^{1/2} e^{(2+\epsilon)t} \quad\text{ with }\quad t\le \frac{1}{2(3+\epsilon)}\log n,
%\]
%where $\epsilon>0$ is an arbitrary small number and $\omega=\omega(n)$ tends to infinity together with $n$. 

%Finally, let us note that for $k=0$ one can choose significantly smaller error function. Indeed, observe that for $k=0$ line~\eqref{eqn:noself} becomes
%\[
%\Mean{\D X_k(i) | \mc{F}_{i-1}} \le -x_k(t(i-1)) + \eps(t(i-1)). 
%\]
%and hence~\eqref{eq:super1} will be
%\begin{equation}\label{{eq:super1b}}
%\eps(t(i-1)) - \eps'(t(i-1)) +  O\left( \frac{|\eps''(\t)|+1}{n} \right) \le 0.
%\end{equation}

%\begin{rem}
%Observe that in the coupon collector's problem the time needed to collect all coupons is the minimum $m$ such that $X_0(m)=0$. For $m = \frac{1}{20}n\log n$ we obtain $X_0(m)\sim n^{19/20}$. 
%The bound $\frac{1}{20}n\log n$ in the above analysis can be improved by choosing $\eps(t)$ more carefully but it is not difficult to see that in any case $m \le (1-\delta)n\log n$ for some constant~$\delta>0$. This yields $X_0((1-\delta)n\log n) = n^{\delta}$ which is still not good enough for the solution of  the coupon collector's problem.
%\end{rem}

\subsection{Self-correction}\label{sec:bins:self}\footnote{This section is a bit more advanced and one can skip it without compromising the understanding of the remaining material.}
Notice that on line \eqref{eqn:noself}, we did something seemingly silly: we were proving the upper bound for the variable $X_k$, and when it came time to replace $X_k$ we used the lower bound. Of course there is a reason for this, since it makes the inequality go the right way. However maybe one wishes to do something smarter. The upper bound on $X_k$ can only fail if gets close to failing first, and when that bound is close to failing is when $\E[\D X_k(i) |  \mc{F}_{i}]$ is even smaller than it  would be otherwise. This ``pushes'' $X_k$ back down even harder. In general, we say that a variable $Y(i)$ is {\em self-correcting} if our expression for $\E[\D Y(i) |  \mc{F}_{i}]$ is decreasing in (has a negative dependence on) $Y(i)$. In this subsection we illustrate how to exploit this situation for a somewhat stronger result (see Figure \ref{fig:opt}(\textsc{\subref{subfig:2}})). 

The idea of self-correction was first used by Telcs, Wormald and Zhou~\cite{TAW2007}. Using the terminology of Bohman, Frieze and Lubetzky~\cite{BFL2015}, we will define a function $0 < \d(t) \le \eps(t)$, a {\em critical interval} 
\[
I(t)=\left[\big(x_k(t) + \d(t)\big)n, \big(x_k(t) + \eps(t)\big)n\right],
\]
and some {\em good but dangerous} events $\mct{E}_{i, j}$. For each $j \le i$ we let  $\mct{E}_{i, j}$ be the event that $\mc{E}_i$ holds and that for each step $j \le i' \le i$, the variable $X_k(i')$ has been in the critical interval $I(t(i'))$. We also define $\mct{E}_{j-1, j} := \mc{E}_{j-1}$. 

We define our super/submartingales to be frozen with respect to the new events $\mct{E}_{i, j}$
\[
X_k^\pm=X_k^\pm(i, j)=\begin{cases} 
X_k(i) - \big(x_k(t)  \pm \eps(t)\big)n & \mbox{if $\mct{E}_{i-1, j}$ holds,}\\
X_k^\pm (i-1) & \mbox{otherwise.}
\end{cases}
\]

Now  we have 
\begin{align*}
\Mean{\D X_k(i) |  \mc{F}_{i}, \mct{E}_{i, j}} &=  -\frac{X_k }{n} + \frac{X_{k-1} }{n}\nonumber\\
&\le -\big(x_k(t ) +\d(t )\big) + \big(x_{k-1}(t ) + \eps(t )\big)\\
&=-x_k(t ) + x_{k-1}(t )  - \d(t ) + \eps(t ), \nonumber
\end{align*}
and so using Taylor's theorem
\begin{align*}
\Mean{\D X_k^+ (i, j)|  \mc{F}_{i}, \mct{E}_{i, j}} &\le -x_k(t ) + x_{k-1}(t )  - \d(t ) + \eps(t )\\
&\qquad  -x_k'(t ) + O\left(\frac{1}{n}\right) - \eps'(t ) + O\left( \frac{|\eps''(\t)|}{n} \right)\\
&= - \d(t ) + \eps(t )- \eps'(t ) +  O\left( \frac{|\eps''(\t)|+1}{n} \right).
\end{align*}
Hence, if
\begin{equation}\label{eq:super2b}
 - \d(t ) + \eps(t )- \eps'(t ) +  O\left( \frac{|\eps''(\t)|+1}{n} \right)\le 0,
\end{equation}
then the $X_k^+(i, j)$ variables are supermartingales. Note that here we can see why we are doing better than our first attempt. When we compare condition \eqref{eq:super2b} to the analogous line~\eqref{eq:super1} from our first attempt, we see that \eqref{eq:super2b} is easier to satisfy. These conditions are like constraints when we choose our error functions $ \d, \eps$ and we satisfy them by making sure $\eps'$ is big enough. But a fast-growing $\eps$ translates to a weaker result so this is exactly where a better inequality yields a better end result. 

Suppose there is some step $i$ where $\mc{E}_i$ fails for the first time due to $X_k(i) > \big(x_k(t) + \eps(t) \big) n$. Consider the step $j \le i$ when $X_k$ first entered the critical interval and such that $X_k$ stayed in the critical interval from step $j$ to step $i-1$. Note that $j\neq 0$, since $X_k(0)=nx_k(0)\notin I(0)$ as $\delta(0)>0$.
This implies that at step $j-1$ we have $X_k(j-1)\le  (x_k(t_{j-1}) + \d(t_{j-1}))n$.
Thus, since $X_k(j) = X_k(j-1)+O(1)$, we get
\begin{align*}
X_k^+(j,j) &= X_k(j) - \big(x_k(t_{j})  + \eps(t_{j})\big)n\\
&= X_k(j-1)  - \big(x_k(t_{j})  + \eps(t_{j})\big)n +O(1)\\
&\le \big(x_k(t_{j-1}) + \d(t_{j-1}) - x_k(t_{j})  - \eps(t_{j})\big)n +O(1)\\
&= -\big(\eps(t_{j}) -\d(t_{j})\big)n + O(1),
\end{align*}
where we have bounded $x_k(t_{j-1}) - x_k(t_{j})$ using Taylor's theorem again, and similarly bounded $\d(t_{j-1}) - \d(t_{j})$. 
Consequently,
\begin{align*}
&\Pr\big(\exists\, i' \le m: \mc{E}_{i'-1} \mbox{ holds but } X_k(i'
) > (x_k(t_{i'}) +\eps(t_{i'}))n\big)\\
&\le \Pr\big(\exists\, j :  X_k^+(m,j) >0\big)\\
&\le \sum_j \Pr\big(X_k^+(m,j) >0\big)\\
&\le \sum_j\Pr\big(X_k^+(m,j) - X_k^+(j,j) > \big(\eps(t_{j}) -\d(t_{j})\big)n + O(1)\big).
\end{align*}
For the Azuma-Hoeffding inequality we also need to make sure that the error function 
\begin{equation}\label{eq:bins_azuma_cond}
|n\D \eps(t_i)| = O(1),
\end{equation}
since we want to have
\[
|\D X_k^+ (i,j)| \le |\D X_k(i)| + |n\D x_k(t_i)| + |n\D \eps(t_i)| \le 1+1+|n\D \eps(t)| = O(1).
\]
Now the Azuma-Hoeffding inequality applied with $m = O(n\log n)$  (recall that in Theorem~\ref{thm:balls} we assume $m \le (\frac{1}{2}-\alpha)n\log n$) and
\[
\lambda = \big(\eps(t_{j}) -\d(t_{j})\big)n + O(1)
\]
yields that the probability of the failure is at most $\exp( -\frac{\lambda^2}{O(m)}).$
Because of the union bound over all $j \le m$,  we need this failure probability to be $o(1/(n\log n))$, so it suffices to choose, for example,  
\[
\delta(t) = n^{-\frac{1}{2}+\frac{\a}{2}} \left(\frac{1}{2}+t\right) \quad \text{ and }\quad \eps(t) = n^{-\frac{1}{2}+\frac{\a}{2}}(1+t),
\]
which also satisfies \eqref{eq:super2b} and~\eqref{eq:bins_azuma_cond} for any arbitrarily small constant $\a>0$.
Note that this error function $\eps(t)$ is much better (that means smaller) than before, since it does not grow exponentially with $t$. Consequently, $t$ (as well $m$) can be chosen to be larger than in the previous section. Finally, observe that $x_k(t) \gg \eps(t)$ even for $t=(\frac{1}{2}-\a)\log n$.

We skip similar calculations for variables $X_k^-$.

\section{Components of  $G(n, m)$ with $k$ vertices}\label{sec:comps}

In this section we prove Theorem \ref{thm:comps}. Recall that our goal is to track the number of components of order~$k$ in $\G(n,\lfloor cn\rfloor)$ for any $1\le k\le \kappa$, where $c$ and $\kappa$ are positive constants. We will track these random variables much the same way as we did in the last section. First, we will guess their trajectories and then rigorously prove that our guess was correct.

Of course determining the sizes of components in a random graph has a lot of history, starting with  Erd\H{o}s and R\'enyi who determined the threshold for a giant component in \cite{ER1960}. Achlioptas posed the following question at a conference: if at each step we are presented with a choice of two random edges and we get to choose which one to add to our graph, can we significantly delay (or accelerate) the emergence of a giant component by implementing an appropriate strategy for how to choose which edge to add? Spencer and Wormald \cite{SW2007} showed that one can delay the giant, and Bohman and Kravitz \cite{BK2006} showed that one can accelerate it. Achlioptas, D'Souza and Spencer \cite{Ach} also provided numerical evidence which seemed to suggest that certain rules would give rise to a {\em discontinuous phase transition},  where the size of the largest component jumps from sublinear to linear in a sublinear number of steps. Surprisingly, Riordan and Warnke~\cite{RW2012} showed for a broad class of strategies  the phase transitions are actually continuous. It is worth noting that the basic ideas we present in our analysis are applicable to Achlioptas processes as well.

\subsection{Determining trajectories heuristically}
Instead of working with $\G(n,m)$ we will consider the Erd\H{o}s-R\'enyi random graph process that starts with $n$ vertices and no edges, and at each step adds one new edge chosen uniformly from the set of missing edges. 
Hence, after precisely $i\le \lfloor cn\rfloor$ edges were revealed, we obtain a graph that is equivalent to $\G(n,i)$.
Let $Y_k(i)$ be the number of components with exactly $k$ vertices in $\G(n,i)$.
Assume heuristically that $Y_k(i) \sim ny_k(t)$ for some function $y_k(t)$, where $t=t_i :=\frac{i}{n}$. Clearly, we must also have $y_1(0)=1$ and $y_k(0) = 0$ for $2\le k\le \kappa$.

Now we will estimate the expected value of $Y_k$. For ease of calculations, we will use the binomial model of random graphs $\G(n,p)$ in which every possible edge occurs independently with probability~$p$. Since in many instances $\G(n,i)$ can be viewed as $\G(n,p)$ with $p=\frac{i}{\binom{n}{2}} \sim \frac{2t}{n}$ (let us emphasize that this is not a rigorous argument), we get
\[
\E[Y_1(i)] \sim n (1-p)^{n-1} \sim n e^{-pn} \sim n e^{-2t},
\]
where the log expansion was used.

Now we consider the case $k\ge 2$. Since in this range of $p=O(1/n)$, the expected number of short cycles is $\binom{n}{k}p^k = O(1)$, $\E[Y_k(i)]$ is asymptotically equal to the expected number of isolated and labelled trees of order $k$. For the latter, first we choose a $k$-subset of vertices~$K$. Next we embed into $K$ a tree~$T$, clearly, with $k-1$ edges. Here we use Cayley's formula to conclude that we have exactly $k^{k-2}$ possible labelled trees~$T$ and each of them appears with probability $p^{k-1} (1-p)^{\binom{k}{2}-(k-1)}$. Finally, observe that $T$ is isolated with probability $(1-p)^{k(n-k)}$. Thus,
\[
\E[Y_k(i)] \sim \binom{n}{k} k^{k-2} p^{k-1} (1-p)^{k(n-k)+\binom{k}{2}-(k-1)}
\sim \frac{n^k}{k!} k^{k-2} p^{k-1} e^{-kpn}
\sim n \cdot \frac{k^{k-2}}{k!} (2t)^{k-1} e^{-2kt}.
\]

Consequently, unifying the above two cases yields that $Y_k(i) \sim ny_k(t)$ for $k\ge 1$, where
\[
y_k(t) = \frac{k^{k-2}}{k!} (2t)^{k-1} e^{-2kt}.
\]

%Notice that here we used a different strategy for guessing trajectories. We did not find any differential equations yet. It will follow from the next section that the corresponding system of differential equations is quite complicated and solving it directly seems to be difficult.

\subsection{Rigorous argument}

We let $\mc{E}_{i'}$ be the event that for each $1\le k\le \kappa$ and each $i \le i'$ we have
\begin{equation}\label{eq:components_goal}
n\big(y_k(t) - \eps(t)\big) \le Y_k(i)\le n\big(y_k(t) + \eps(t)\big),
\end{equation}
where 
\[
\eps(t) = n^{-1/3} e^{6\kappa^3 t}.
\]
Considering the values of functions $y_k(t)$ and $\eps(t)$ we can even get tight concentration on our random variables when the number of edges is of order $n\log n$ (since $\eps(t) =o( y_k(t))$ when $i=cn\log n$ for a small enough constant $c>0$). But we only interested in $i=O(n)$.

Define variables 
\[
Y_k^\pm=Y_k^\pm(i)=\begin{cases} 
Y_k(i) - \big(y_k(t)  \pm \eps(t)\big)n & \mbox{if $\mc{E}_{i-1}$ holds,}\\
Y_k^\pm (i-1) & \mbox{otherwise.}
\end{cases}
\]
We show that the variables $Y_k^+(i)$ are supermartingales and $Y_k^-(i)$ are submartingales for $i\le \lfloor cn\rfloor$.

Assume that precisely $i$ edges were revealed during the Erd\H{o}s-R\'enyi process.
The probability of choosing any particular $(i+1)$st edge that has not been chosen yet is 
\[
 \frac{1}{\binom{n}{2} - i} = \frac {2}{n^2} \left(1+O\left(\frac{1}{n}\right)\right).
\]

We will analyze the expected one-step change of $Y_k(i)$. Assume that the graph $G(n,i)$ is already generated. Now a new edge $e$ is added uniformly at random from the set of missing edges. First we discuss the negative contribution to $\E[\D Y_k(i) |  \mc{F}_{i}]$, meaning the expected number of components of order $k$ that are lost in one step. To lose such components, we must have either exactly one of the endpoints of $e$ in a component of order~$k$ (here we have $kY_k\cdot (n-kY_k)$ choices, each of which would decrease $Y_k$ by 1) or both endpoints of $e$ belong to different  components of order~$k$ (yielding $\binom{Y_k}{2}k^2$ choices that decrease $Y_k$ by 2). Thus, the negative contribution is 
\begin{align*}
&-\left(kY_k(n-kY_k) + 2\binom{Y_k }{2}k^2 \right) \cdot \frac {2}{n^2} \left(1+O\left(\frac{1}{n}\right)\right) = -\frac{2kY_k}{n} +O\left(\frac{1}{n}\right),
\end{align*}
where we used the fact that trivially $Y_k \le n$ and so $\frac{Y_k}{n^2} = O(\frac{1}{n})$.

For the positive contribution, $e$ must be between two components of order $j$ and $k-j$ for $1\le j\le k-1$ giving $jY_j \cdot(k-j)Y_{k-j} $ choices for $e$. Hence, the positive contribution is
\[
\frac{1}{2}\sum_{j=1}^{k-1} jY_j (k-j)Y_{k-j} \cdot \frac {2}{n^2} \left(1+O\left(\frac{1}{n}\right)\right) = \sum_{j=1}^{k-1} j(k-j) \frac{Y_j }{n}\frac{Y_{k-j} }{n}  +O\left(\frac{1}{n}\right).
\]
The $\frac{1}{2}$ factor above is needed since in the above sum each possible edge $e$ is counted exactly twice.

 Thus,
 \[
 \Mean{\D Y_k(i) |  \mc{F}_{i}} =  -\frac{2kY_{k} }{n} + \sum_{j=1}^{k-1} j(k-j) \frac{Y_j }{n}\frac{Y_{k-j} }{n} + O\left(\frac{1}{n}\right).
 \]
(We note in passing that from here we can see that $Y_k$ is self-correcting. We will not exploit it here, but the interested reader might try using it to improve the error functions we will get.) Therefore, we have (assuming as before that the good event $\mc{E}_{i}$  holds) 
\begin{align*}
\Mean{\D Y_k(i) |  \mc{F}_{i}} &\le -2k\big(y_k(t) - \eps(t)\big) + \sum_{j=1}^{k-1} j(k-j) \big(y_j(t) + \eps(t)\big)\big(y_{k-j}(t) + \eps(t)\big) + O\left(\frac{1}{n}\right)\\
&= -2ky_k(t) + \sum_{j=1}^{k-1} j(k-j)y_j(t) y_{k-j}(t) \\
&\qquad\qquad+ 2k\eps(t)  + \eps(t) \cdot \sum_{j=1}^{k-1} j(k-j) \big(y_j(t) + y_{k-j}(t)+\eps(t)\big)  + O\left(\frac{1}{n}\right).
\end{align*}
Now since $y_j + y_{k-j}+ \eps \le 3$ we have 
\[
\sum_{j=1}^{k-1} j(k-j) \big(y_j(t) + y_{k-j}(t)+\eps(t)\big) \le 3k^3,
\]
and so we finally get 
\[
\Mean{\D Y_k(i) |  \mc{F}_{i}} \le -2ky_k(t) + \sum_{j=1}^{k-1} j(k-j)y_j(t) y_{k-j}(t) + 5k^3\eps(t) + O\left(\frac{1}{n}\right).
\]
Consequently, Taylor's theorem yields
\begin{align*}
\Mean{\D Y_k^+ (i)|  \mc{F}_{i}, \mc{E}_{i}} 
&= -2ky_k(t) + \sum_{j=1}^{k-1} j(k-j)y_j(t) y_{k-j}(t) + 5k^3\eps(t) \\
&\qquad\qquad - y_k'(t)  - \eps'(t) + O\left(\frac{1}{n}\right)\\
&=5k^3\eps(t)- \eps'(t) + O\left(\frac{1}{n}\right) \le 0,
\end{align*}
due to the following claim and choice of $\eps(t)$.

\begin{claim}\label{claim:components}
For any positive integer $k$,
\[
y_k'(t) = -2ky_k(t) + \sum_{j=1}^{k-1} j(k-j)y_j(t) y_{k-j}(t).
\]
\end{claim}

Before we prove this claim note that we obtained a system of first-order nonlinear differential equations and it is not obvious how to solve it without predicting the solutions first. One can easily solve for $y_0$ and then use it to get $y_1$ and so on but how to obtain a closed-form expression for $y_k$ this way is not clear.  Once we observe that our functions $y_k$ are a solution to the system of differential equations above, it is easy to argue that they are the unique solution.  Indeed this follows by induction since the expression for $y_k'$ involves only $y_0, \ldots, y_k$ and if we already know $y_0, \ldots, y_{k-1}$ then we have a first-order linear differential equation in $y_k$. 

\begin{proof}[Proof of Claim~\ref{claim:components}]
For $k=1$ it is trivial. For $k\ge 2$ we get that 
\begin{align}\label{eq:checking_diff}
-2ky_k(t) + \sum_{j=1}^{k-1} j(k-j)&y_j(t) y_{k-j}(t)
= -2\frac{k^{k-1}}{k!} (2t)^{k-1} e^{-2kt}\notag \\
&\qquad\qquad\qquad\qquad\qquad + (2t)^{k-2} e^{-2kt} \sum_{j=1}^{k-1} \frac{j^{j-1} (k-j)^{k-j-1}}{j!(k-j)!}\notag \\
&= -2\frac{k^{k-1}}{k!} (2t)^{k-1} e^{-2kt} + (2t)^{k-2} e^{-2kt} \frac{1}{k!} \sum_{j=1}^{k-1} \binom{k}{j} j^{j-1} (k-j)^{k-j-1}\notag \\
&= -2\frac{k^{k-1}}{k!} (2t)^{k-1} e^{-2kt} + (2t)^{k-2} e^{-2kt} \frac{1}{k!} \cdot 2(k-1)k^{k-2},
\end{align}
where the latter follows from~\eqref{eq:diff_eq} (see below). Finally, note that the product rule implies that $y_k'(t)$ is exactly given by \eqref{eq:checking_diff}.   

It remains to show that for any positive integer~$k$ the following holds:
\begin{equation}\label{eq:diff_eq}
\sum_{j=1}^{k-1} \binom{k}{j} j^{j-1} (k-j)^{k-j-1}= 2(k-1)k^{k-2}.
\end{equation}
We will count the labelled trees of order $k$ on a set $K$ of size~$k$. First choose a subset $J\subseteq K$ of size $j$ ($\binom{k}{j}$ choices), next a labelled  tree $T_1$ on $J$ ($j^{j-2}$ choices due to Cayley's formula) and a tree $T_2$ on $K\setminus J$ ($(k-j)^{k-j-2}$  choices). Finally choose an edge $e$ between $T_1$ and $T_2$ ($j(k-j)$ choices). Clearly $T_1$ and $T_2$ together with $e$ is a labelled tree on $K$. In order to finish the proof observe that each tree on $K$ can be obtained by the above procedure in exactly $2(k-1)$ ways.
\end{proof}

Now we are going to use the Azuma-Hoeffding inequality. First observe (similarly as in Section~\ref{sec:bins:rigorous}) that 
\[
|\D Y_k^+ (i)| \le |\D Y_k(i)| + |n\D y_k(t_i)| + |n\D \eps(t_i)| =O(1),
\]
since $|\D Y_k(i)| \le 2$. Thus for $m=O(n)$, 
\begin{align*}
\Pr\big(\exists\; i'\le m: \mc{E}_{i'-1} \mbox{ holds but }& Y_k(i') > n(y_k(t_{i'}) + \eps(t_{i'}))\big) = \Pr(Y_k^+(m) >0)\\
& = \Pr(Y_k^+(m) - Y_k^+(0) > n\eps(0)) \le \exp\left( -\frac{(n\eps(0))^2}{2C^2 m} \right) = o(1),
\end{align*}
yielding that the upper bound in~\eqref{eq:components_goal} holds a.a.s..

One can also show by providing symmetric calculations that the variables $Y_k^-(i)$ are submartingales and also~\eqref{eq:azuma} applied to supermartingales $-Y_k^-(i)$ yields that a.a.s.\  the lower bound in~\eqref{eq:components_goal} holds.

\section{Almost perfect matchings in regular graphs}\label{sec:match}

In this section we prove Theorem \ref{thm:match}. We assume that $G=G_n$ is a $d$-regular graph of order $n$ with $d=d_n$ satisfying $d \gg \log n$. To find a large matching, we consider the random process that chooses one edge at every step, chosen randomly from all edges in $G$ that do not share any endpoints with previously chosen edges (or we halt if no such choice is possible). A more formal description follows. 

\subsection{Random greedy algorithm}
 We consider the following random greedy process that forms a matching~$M$. The process builds a matching step by step. At step $i$ we call the current matching $M(i)$, the set of unmatched vertices $V(i)$, and we call the graph induced on unmatched vertices $G(i) = G[V(i)]$. Say our starting graph is $G=(V, E)$.  We start with $M(0)=\emptyset$, $G(0)=G$ and $V(0)=V$. In step $i\ge 1$ an edge $e_i$ is chosen uniformly at random from $G(i-1)$ and added to $M(i-1)$ to form a matching $M(i)$. Then graph $G(i)$ is obtained from $G(i-1)$ by setting $V(i) = V(i-1)\setminus e_i$ and $G(i) = G[V(i)]$.
 %deleting from $G(i-1)$ all edges with an endpoint in~$e_i$.

The goal is to estimate the size of $M$ at the end of the process. We show that a.a.s.\  the above algorithm finds a matching of size at least $(\frac{1}{2}-\a)n$ for a function $\alpha = \alpha(n)=o(1)$. In other words, all but at most $2\a n$ vertices are matched.

\subsection{Determining trajectories heuristically}
We will show that at each step $i <(\frac{1}{2}-\a)n$, $G(i)$ is nonempty. In order to do it, we will keep track of vertex degrees. Let 
\[
t = t_i := \frac{i}{n}
\]
be a continuous time parameter.
We anticipate that $G(i)$ resembles a subgraph of $G$ induced by a random subset of the vertices where each vertex
is included independently with probability
\[
 \frac{n-2i}{n} = 1 - \frac{2i}{n} = 1-2t =:p=p(t),
\]
and note that if we only track the process for $(\frac{1}{2}-\a)n$ steps, then we always have 
\begin{equation}\label{eq:matchings:p_alpha}
p \ge 2\a.
\end{equation} 
We guess that the degree of vertex $v$ in $G(i)$, denoted by $D_v(i)$, should be about~$dp$.

\subsection{Rigorous argument}
Recall that $G=(V,E)$ is a $d$-regular graph of order $n$ with $d=d_n$ satisfying $d \gg \log n$.  For convenience, we will define $D_v$ even for matched vertices~$v$; we simply let $D_v$ be the number of unmatched neighbors of $v$. We let our good event $\mc{E}_{i'}=\mc{E}_{i'}(v)$ be that for all $i\le i'$ and all $v\in V$ we have
\begin{equation}\label{eq:matching:goal}
dp(t)  - \eps(t ) \le D_v(i) \le dp(t)  + \eps(t ),
\end{equation}
for some error function $\eps< dp$ to be determined later. 
We will show that a.a.s.\  the good event $\mc{E}_{(1/2 - \a)n}$ holds. This will imply that up to step $(1/2 - \a)n$ we have  $D_v \ge dp(t)-\eps(t) > 0$  
and so there are still edges in $G(i)$, meaning the process keeps running. Thus, we get a matching of size at least $\left(1/2-\a \right)n$.

We remark that the analysis we provide here can be improved and generalized (see Bennett and Bohman~\cite{BB2019}).

As in the previous examples we define
\[
D_v^\pm=D_v^\pm(i)=\begin{cases} 
D_v(i) - \big(dp(t)  \pm \eps(t )\big) & \mbox{if $\mc{E}_{i-1}$ holds,}\\
D_v^\pm (i-1) & \mbox{otherwise.}
\end{cases}
\]
We then show that these variables are super/submartingales.

We calculate the expected one-step change of $D_v(i)$. Assume that $G(i) = (V(i), E(i))$ was already generated and now a new edge $e$ is chosen uniformly at random from this graph and added to~$M(i)$. The degree of a vertex $v$ can only be affected by $e$ if $e$ has an endpoint in the neighborhood of $v$, $N(v)$. %Actually there are two cases: $e$ has only one endpoint in $N(v)$ (and so the degree of $v$ drops by one) or $e$ has two endpoints in $N(v)$ (dropping the degree of $v$ by two). The number of possible choices for $e$ in the first and the second case is 
% \[
% \sum_{u\in N_{G(i)}(v)} d_u - 2|\{e : e\subseteq N(v) \}| \quad\text{ and }\quad  |\{e : e\subseteq N(v) \}|,
% \]
Thus,
\begin{align}\label{eqn:1scD}
\Mean{\D D_v(i) | \mc{F}_{i}}  &= -\frac{ \sum_{e \in E(i)}  |N_{G(i)}(v) \cap e|}{|E(i)|} = -\frac{ \sum_{u\in N_{G(i)}(v)} D_u}{\frac{1}{2}\sum_{w\in V(i)}  D_w},
\end{align}
where in the second equality the denominators are clearly equal, and the numerators are equal because both sums count edges with one endpoint in $N_{G(i)}(v)$ once, and edges with both endpoints there twice. (From \eqref{eqn:1scD} we can see that $D_v$ is self-correcting, since the sum in the numerator has $D_v$ many terms. We leave it as an exercise for the reader to exploit self-correction and improve the result we will get here.)

Using \eqref{eqn:1scD}, we have
\begin{align*}%\label{eq:matching_Ddv}
\Mean{\D D_v(i) | \mc{F}_{i}, \mc{E}_{i}}  &\le -\frac{(dp -\eps)(dp -\eps)}{\frac{1}{2} np(dp+\eps)}= -\frac{2d}{ n} \cdot \frac{\left( 1-\frac{\eps}{dp} \right)^2}{1+\frac{\eps}{dp}}\notag \\
&=-\frac{2d}{n} \cdot \left(1- \frac{3\eps}{dp} + \frac{4(\frac{\eps}{dp})^2}{1+\frac{\eps}{dp}}\right)
=-\frac{2d}{n} + \frac{6\eps}{np} + O\left(\frac{\eps^2}{ndp^2}\right),
\end{align*}
since we are assuming that $1\le \eps \le dp$.
An attentive reader can note that the big-O error term it is not needed here, since for any number $x>0$, we have $\frac{(1-x)^2}{1+x} \ge 1-3x$. However, at the end we like to claim that almost everything that we do for supermartingales is ``symmetric'' for the submartingales. Since the big-O term will be necessary for the submartingales, we will include it in our calculations for the supermartingale and show that it is insignificant. We have
\[
\Mean{\D D_v^+ (i)|  \mc{F}_{i}, \mc{E}_{i}} \le -\frac{2d}{n} + \frac{6\eps}{np}-\frac{dp'(t)}{n} - \frac{\eps'}{n} + O\left(\frac{\eps^2}{ndp^2} + \frac{\eps''(\t)}{n^2}\right),
\]
where the big-O term now also contains an error from Taylor's theorem (note that we only have an $\eps''$ term since $p''=0)$.
Furthermore, since $p'(t) = -2,$
we obtain that
\begin{equation}\label{eqn:supd}
\Mean{\D D_v^+ (i)|  \mc{F}_{i}, \mc{E}_{i}} \le \frac{6\eps}{np} - \frac{\eps'}{n} + O\left(\frac{\eps^2}{ndp^2} + \frac{\eps''(\t)}{n^2}\right) .
\end{equation}

Now we discuss how to choose $\eps(t)$. Since we would like the above to be negative, we first think of making 
\[
\eps' > \frac{6 \eps}{p}
\]
``with room to spare'', in the hopes that this extra room will beat the big-O term. We see that we can get $\eps' = 8\eps/p$ if we choose 
\[
\eps(t) = s (p(t))^{-4}
\]
for any scaling factor $s=s(n)$ not depending on $t$.  We will be forced to choose $s \to \infty$ for this error bound to hold a.a.s.. But in order to be a useful bound (i.e.\ to guarantee that $D_v$ remains positive) we must have $dp > \eps(t) =sp^{-4}$, so we will only prove the algorithm keeps running so long as $p > \rbrac{s/d}^{1/5}$. Thus, since $p\ge 2\alpha$ (cf.~\eqref{eq:matchings:p_alpha}), we will require 
\begin{equation}\label{eqn:alphareq1}
\a \gg \rbrac{\frac{s}{d}}^{1/5}.
\end{equation}

 Now we can see that  \eqref{eqn:supd} is negative so long as $p \ge 2\a\gg \rbrac{\frac{s}{d}}^{1/5} \gg 1/n$.
 % (which we will see holds easily and is less restrictive than \eqref{eqn:alphareq1}). 
 Indeed, the right-hand side of \eqref{eqn:supd} becomes
\[
-\frac{2s}{np^5} + O\left(\frac{s^2}{ndp^{10}} + \frac{s}{n^2(p(\t))^6}\right)
= \frac{s}{np^5} \left(-2 + O\left(\frac{s}{dp^{5}} + \frac{1}{np}\right)  \right) < 0,
\]
since $\t\in [t_i,t_{i+1}]$ and hence $p(\t) = 1-2\t\ge 1-2t_{i+1} = 1-2t_i - 1/n$. 
So we just showed that the $D_v^+$ are supermartingales. 
Symmetric calculations show that the $D_v^-$ are submartingales.

%Now we provide symmetric calculations to show that the variables $D_v^-$ are submartingales.
%Notice that
%\begin{align*}
%\Mean{\D D_v(i) |  \mc{F}_{i}}  &= -\frac{\sum_{u\in N_{G(i)}(v)} (d_u-1)}{\frac{1}{2}\sum_{w\in V(G(i))}  d_w} \ge -\frac{(dp +\eps)^2}{\frac{1}{2} (n-2tn)(dp-\eps)}\\
%&= -\frac{dp}{\frac{1}{2} n(1-2t)} \cdot \frac{\left( 1+\frac{\eps}{dp} \right)^2}{1-\frac{\eps}{dp}}
%\ge -\frac{dp}{\frac{1}{2} n(1-2t)} \cdot \left(1+4\frac{\eps}{dp}\right)\\
%&= -\frac{dp}{\frac{1}{2} n(1-2t)} - \frac{8\eps}{(1-2t)n}
%\end{align*}
%and consequently 
%\[
%\Mean{\D D_v^- (i)|  \mc{F}_{i}} 
%\ge -\frac{dp}{\frac{1}{2} n(1-2t)}- \frac{8\eps}{(1-2t)n} + \frac{2d}{n} + \frac{\eps'}{n}
% + O\left( \frac{\eps''}{n^2}\right)  = O\left( \frac{\eps''}{n^2}\right) \ge 0.
%\]

One can easily check that the Azuma-Hoeffding inequality is not strong enough for this problem. Indeed, in this case our failure probability would be 
\begin{align*}
\Pr(D_v^+(tn) - D_v^+(0) > \eps(0))\le \exp\left( -\frac{(\eps(0))^2}{2C^2 tn} \right)
=\exp\left(-\frac{s^2}{O(n)} \right), 
\end{align*}
which requires $s\gg \sqrt{n}$ to be small. But $p > \rbrac{s/d}^{1/5}$ yields that $d\ge s$ and so $d\gg \sqrt{n}$. Thus, the approach with  the Azuma-Hoeffding inequality is completely useless if $d\le \sqrt{n}$. We could settle for proving this result only for $d \gg \sqrt{n}$, but we would like to do better. Therefore, we are going to use a different martingale concentration inequality to bound the failure probability. 
How should it be different? The problem with the Azuma-Hoeffding inequality here is the parameter~$C$ which we are forced to take to be a constant, essentially because $D_v$ (and therefore $D_v^{\pm}$) can change by a constant at any step. But when the graph is very sparse (the case we cannot handle with the Azuma-Hoeffding inequality), the event that $D_v$ changes at all is rare and actually its expected one-step change is $o(1)$.  For variables such as this, which experience relatively large but rare one-step changes, Freedman's inequality~\eqref{eq:freedman} is often helpful.

Before we apply Freedman's inequality we need to do some preparation. Note that
we can take $C=O(1)$, since 
\[
|\D D_v^+ (i)| \le |\D D_v(i)| + |d\D p(t_i)| + |\D\eps(t_i)| \le 2 + d\cdot \frac{2}{n} + O\left(\frac{s}{n\a^5}\right) \le O(1)
\] 
(and note the same bound holds for $|\D D_v^- (i)|$). Using the above bound, next we bound the one-step variance. Outside the good event $D_v^+$ is frozen so $\Var[\D D_v^+ (i)|  \mc{F}_{i}, \neg \mc{E}_{i}]=0$. Inside the good event we have
\[
\Var[\D D_v^+ (i)|  \mc{F}_{i}, \mc{E}_{i}] \le \E[(\D D_v^+ (i))^2|  \mc{F}_{i}, \mc{E}_{i}] \le C \cdot \E[\;|\D D_v^+ (i)|\; |  \mc{F}_{i}, \mc{E}_{i}]
\]
for some absolute constant $C > 0$.
When applying Freedman's inequality in this context, it is common to use extremely simple bounds on the variance as above. Now we try to obtain a crude bound on $\E[\;|\D D_v^+ (i)|\; |  \mc{F}_{i}, \mc{E}_{i}]$. In the good event we have
\begin{align*}
|\D D_v^+ (i)| &= \abrac{\D D_v (i)  -\frac{dp'(t)}{n} - \frac{\eps'}{n} + O\left(\frac{\eps''(\t)}{n^2}\right)}\\
&= \abrac{\D D_v (i) +\frac{2d}{n}  -\frac{2s}{np^5} + O\left(\frac{s}{n^2p^6}\right)}
\le \abrac{\D D_v (i)}  + O\rbrac{\frac{d}{n}}.
\end{align*} 
For $k\in\{0,1,2\}$, let $x_k$ be the number of edges with exactly $k$ endpoints in $N_{G(i)}(v)$. Since, trivially, $x_0 \le 2ndp$ (the ``2'' could be replaced with ``$1+o(1)$'' actually) and similarly  $x_1, x_2 \le 2d^2p^2$ and $d/n\le 1$, we have 
\begin{align*}
  \E[\;|\D D_v^+ (i)|\; |  \mc{F}_{i}, \mc{E}_{i}]& \le \frac{\sum_{e \in E(i)} \abrac{ \frac{2d}{n} - |N_{G(i)}(v) \cap e| }}{|E(i)|}\\
  &= \frac{ x_0 \cdot O\rbrac{\frac{d}{n}} + x_1 \cdot \rbrac{1+O\rbrac{\frac{d}{n}}}+ x_2 \cdot \rbrac{2+O\rbrac{\frac{d}{n}}} }{|E(i)|}\\
  & = O\rbrac{\frac {d^2 p + d^2 p^2}{ndp}} = O\rbrac{\frac {d  }{n}}.
\end{align*}

%where we are claiming that $\E[\;|\D D_v^+ (i)|\; |  \mc{F}_{i}]=O\left(\frac{s}{n\a^5}\right)$. While it is clear from \eqref{eqn:supd} that $\E[\D D_v^+ (i) |  \mc{F}_{i}]=O\left(\frac{s}{n\a^5}\right)$, we leave it as an exercise to show that the same bound holds for $\E[\;|\D D_v^+ (i)| \;|  \mc{F}_{i}]$. 
Consequently, for Freedman's inequality we can choose the parameter $b$ as follows:
\[
\sum_{i} \Var[ \D D_v^+ (i) |  \mc{F}_{i}] = O(n) \cdot O\rbrac{\frac{d}{n}} = O(d) 
\]
 so that we can choose $b:=O(d)$  is an upper bound on the sum of one-step variances over all steps.
Now since $D_v^+ (0) = -s$, we set $\lambda = s$. Thus, Freedman's inequality implies that the probability that the good event $\mc{E}_{i'}$ fails due to the condition ``$D_v(i') > dp(i') + \eps(t_{i'})$'' is at most
\begin{align*} 
\Pr&\big(\exists\; i'\le m: \mc{E}_{i'-1} \mbox{ holds but } D_v(i') > dp(t_{i'}) + \eps(t_{i'})\big)\\
&= \Pr(D_v^+(m) - D_v^+(0) > \lambda)\\
&\le \Pr(\exists i'\le m: V_{i'} \le b \text{ and } D_v^+(i') - D_v^+(0)>\lambda) \\
&\le \exp\left(-\frac{\lambda^2}{2(b+C\lambda) }\right)
= \exp\left(-\frac{s^2}{O(d)+O(s) }\right) = \exp\rbrac{-\Omega\rbrac{\frac{s^2}{d }}}.
\end{align*}
Let 
\begin{equation}\label{eqn:sdef}
    s = K\sqrt{d \log n}
\end{equation}
for some sufficiently large constant $K>0$. Then we can beat the union bound over all vertices since
\begin{align*}
    \Pr\big(\exists\;v,  i'\le m: \mc{E}_{i'-1} &\mbox{ holds but } D_v(i') > dp(t_{i'}) + \eps(t_{i'})\big)\\
    &\le n  \exp\rbrac{-\Omega\rbrac{\frac{s^2}{d }}} = n \exp(-\Omega(K^2 \log n)) = o(1).
\end{align*}

Hence, a.a.s.\  for every vertex $v$ the upper bound in~\eqref{eq:matching:goal} is valid. Note also that in light of our choice of $s$ in \eqref{eqn:sdef}, we may now look back to \eqref{eqn:alphareq1} to see how strong our result is. Our analysis holds until a step when the proportion of unmatched vertices, $2\alpha$, is a large constant times 
\[
 \rbrac{\frac{s}{d}}^{1/5} = K^{1/5} \left(\frac{\log n}{d}\right)^{1/10}.
\]
Thus, we can take $\alpha:=\rbrac{\frac{s}{d}}^{1/10}  = o(1)$, since we are assuming that $\frac{d}{\log n} \rightarrow \infty$. 

Finally, one can apply Freedman's inequality to the supermartingale $-D_v^-$ to show that the probability that the good event $\mc{E}_i$ fails due to the condition ``$D_v(tn) < dp(t) - \eps(t)$'' is also small.

\section{Further remarks}

Our goal in writing this paper is to make the differential equation method more accessible to other researchers because there have been several very nice applications of it, and we believe there are still many more to be had in the future. However, the most famous applications involve a potentially intimidating or overwhelming amount of details and make the method look harder than it really is. So we attempted to show the method using the most simple nontrivial examples. 

Now, for the interested reader we will suggest some papers to read next. Our goal is to suggest papers that are less gentle than this one, but still not too intimidating. Though this need not be one's first priority, it is valuable to see a proof of a black box theorem and for that we recommend that the reader see Warnke's short paper \cite{Warnke2020}. D\'iaz and Mitsche's ``cook-book'' survey \cite{DM2010} has many applications of Wormald's original black box theorem, so an interested reader might try reading a few of those. 

However, many of the most interesting applications of the differential equation method are not suitable for any known black box theorem. They require their own analysis typically using martingales and resembling the approach we took in this paper. For a nice short paper like that the reader should see Bohman, Frieze and Lubetzky's first paper on the triangle-removal process \cite{BFL10}. A good next step would be to read Bohman's original paper on the triangle-free process \cite{B2009}. The challenging part of that paper is to understand how Bohman bounded the independence number of the graph, and from that part the reader may begin to see how the analysis can benefit from augmenting the family of tracked variables to be much larger than what might na\"ively seem necessary. We further suggest the following (what we consider) easy paper \cites{BB2019}, medium papers \cites{BBoh2016, Warnke2014, BW2019}, and hard papers \cites{BFL2015, BK2010}. 

 \newcommand{\noop}[1]{}

\end{document}